\documentclass[10pt]{amsart}

\usepackage{amsmath,amssymb,verbatim,graphicx,amsthm,color}
\usepackage[top=1in, bottom=1in, left=1in, right=1in]{geometry}
\usepackage{cite}
\allowdisplaybreaks[1]
\usepackage[sc]{mathpazo}
\usepackage[T1]{fontenc}

\newcommand{\vertiii}[1]{{\left\vert\kern-0.25ex\left\vert\kern-0.25ex\left\vert #1 
    \right\vert\kern-0.25ex\right\vert\kern-0.25ex\right\vert}}
\newcommand{\braciii}[1]{{\left[\kern-0.25ex\left[\kern-0.25ex\left[ #1 
		\right]\kern-0.25ex\right]\kern-0.25ex\right]}}
\newcommand{\norm}[1]{\left\lVert#1\right\rVert}

\newcommand{\ZZ}{\mathbb{Z}}

\theoremstyle{plain}

\newtheorem{theorem}{Theorem}[section]

\newtheorem{prop}[theorem]{Proposition}
\newtheorem{lemma}[theorem]{Lemma}
\newtheorem{cor}[theorem]{Corollary}
\newtheorem{mydef}[theorem]{Definition}

\usepackage{setspace}
\begin{document}
\title{Coboundaries of nonconventional ergodic averages}
    \author{I. Assani}
\thanks{Department of Mathematics, UNC Chapel Hill, NC 27599, assani@email.unc.edu}
 \begin{abstract}
 Let $(X,\mathcal{A}, \mu)$ be a probability measure space and let $T_i,$ $1\leq i\leq H,$ be  invertible bi measurable measure preserving transformations on this measure space. We give a sufficient condition for  the product of $H$ bounded functions $f_1, f_2, ..., f_H$ to be a coboundary.  This condition turns out to be also necessary  when one seeks bounded coboundaries.
 \end{abstract} \maketitle

  \section{Introduction}

The purpose of this short article is to answer a question brought to our attention by S. Donoso (\footnote{He indicated that this question was mentioned to him by J.P. Conze and Y. Kifer}) during the 2017 ETDS workshop held at Chapel Hill. \\
To this end we refine the setting in \cite{Ass2}.
\begin{mydef}
A probability measure preserving system  $(X, \mathcal{F}, \mu, T_1, T_2, \ldots, T_H)$ is a combination of a probability measure space $(X, \mathcal{F}, \mu)$ and $T_i, 1\leq i\leq H$ bi-measurable  invertible measure preserving maps acting on this probability space. 
\end{mydef}

 Given a probability measure preserving system $(X, \mathcal{F}, \mu, T_1, T_2, \ldots, T_H)$, $\mu_\Delta$ is the diagonal measure on $X^H$, $\Phi = T_1 \times T_2 \times \cdots \times T_H$, and $\nu$ is the diagonal-orbit measure of $\Phi$, i.e.
\[\nu(A) = \frac{1}{3} \sum_{n \in \ZZ} \frac{1}{2^{|n|}} \mu_\Delta(\Phi^{-n}A). \] 	
We note that $\nu$ is nonsingular, since $\frac{1}{3} \nu(A) \leq \nu(\Phi^{-1}A) \leq 2\nu(A).$.
\begin{mydef} 
The diagonal orbit system of the probability measure preserving system $(X, \mathcal{F}, \mu, T_1, T_2, \ldots, T_H)$ is the system $(X^H, \mathcal{F}^H, \nu, \Phi).$
\end{mydef} 
\noindent{\bf Remarks}
 \begin{enumerate}
 \item The maps $T_i$ do not necessarily commute.
 \item The nonsingularity of $\Phi$ with respect to $\nu$ implies the following simple but \bf{key lemma} (This lemma does not seem to hold when one replaces $\nu$  with the diagonal measure $\mu_{\Delta}$ on $\big(X^H, \mathcal{F}^H),$ defined by  the equation $\int F(x_1,x_2, \ldots,x_H) d\mu_{\Delta} = \int F(x,x,\ldots, x) d\mu$ \big). 
 \end{enumerate}
 \begin{lemma}\label{L1}
 Let $F_n$ be a sequence of measurable functions defined on $X^H.$ If $F_n$ converges $\nu$ a.e. then 
 the sequence $G_n= F_n\circ\Phi $ converges $\nu$ a.e. as well. 
 \end{lemma}
 \begin{proof}
   Let $A = \{z\in X^H; F_n(z)\, \text{converges}\} $ and $ B=\{z\in X^H; G_n(z)\, \text{converges}\}.$ 
   We have $B = \Phi^{-1}(A). $ Therefore if $\nu(A^c) = 0$  we have $\nu(\Phi^{-1}(A^c)) = 0$ by the non singularity of $\nu.$
 \end{proof}
We wish to prove the following proposition.
\begin{prop}\label{coboundary}
	Let $(X, \mathcal{F}, \mu, T_1, \ldots, T_H)$ be a measure preserving system, and $f_1, f_2, \ldots f_H \in L^\infty(\mu)$ and $1\leq p <\infty.$ 
	\begin{enumerate}
		\item  Let us assume that the supremum of the nonconventional ergodic sums is $L^p$-bounded, i.e.
		\[ \sup_{N} \norm{\sum_{n= 1}^N \prod_{i=1}^H f_i \circ T_i^{n}  }_{L^p(\nu)} < \infty. \]
		\item Then the product of the functions is a coboundary in $L^p(X^H, \nu)$, i.e. if $\Phi = T_1 \times T_2 \times \cdots \times T_H$, there exists $V\in L^p(X^H, \nu)$ such that
		\[\bigotimes_{i=1}^H f_i = V - V \circ \Phi,  \text{$\nu$-a.e. }  
		 \]
		Therefore, for $\mu_{\Delta}$-a.e. $(x_1, x_2,...,x_H) \in X^H$, we have
		\[f_1(x_1)f_2(x_2) \cdots f_H(x_H) = V(x_1, x_2, \ldots, x_H) - V(T_1x_1, T_2x_2, \ldots, T_Hx_H). \]
	\end{enumerate}
\end{prop}
We give only the proof for the case $p=1.$   We use the following a.e.-convergence result obtained by Koml\'os in 1967. When $1<p<\infty$ the reflexivity of $L^p(\nu)$ allows to bypass this lemma.  
For $p = \infty$ the assumptions (1) and (2) in the statement of  Theorem \ref{coboundary} are equivalent . We state it separately as a corollary.
\begin{lemma}[\cite{Komlos}]\label{Komlos}
	Let $(X, \mathcal{F}, \mu)$ be a probability measure space, and $(g_n)$ be a sequence in $L^1(\mu)$. Assume that $\liminf_n \norm{g_n}_{L^1(\mu)} < \infty$. Then there exists a subsequence $(g_{n_k})_k$ and a function $g \in L^1(\mu)$ such that for $\mu$-a.e. $x \in X$, 
	\[\lim_{K \to \infty} \frac{1}{K} \sum_{k=1}^K g_{n_k}(x) = g(x). \]
\end{lemma}
	\begin{proof}[Proof of Proposition $\ref{coboundary}$]
		We show that the  techniques used for an invariant measure can be applied  to our current  nonsingular setting. 
		The assumption made tells us that, if $F = f_1 \otimes f_2 \otimes \cdots \otimes f_d$, we have
		\[ \lim_{N \to \infty} \norm{\frac{1}{N} \sum_{n=1}^N F \circ \Phi^n}_{L^1(\nu)} = 0. \]
		Therefore, there exists a subsequence $N_k$ of natural numbers such that 
		\[\lim_{N_k \to \infty} \frac{1}{N_k} \sum_{n=1}^{N_k} F \circ \Phi^n(z) = 0 \]
		for $\nu$-a.e. $z \in X^H$. Because $\Phi$ is nonsingular with respect to $\nu$, we also know that
		\[\lim_{N_k \to \infty} \frac{1}{N_k} \sum_{n=1}^{N_k} F \circ \Phi^{n+1}(z) = 0 \]
		for $\nu$-a.e. $z \in X^H$ for the same subsequence $(N_k)$. Set 
		\[D_{N_k} = \frac{1}{N_k} \sum_{n=1}^{N_k} \sum_{j=0}^{n-1} F \circ \phi^j. \]
		Since $\sup_n\norm{\sum_{j=0}^{n-1} F \circ \phi^j}_{L^1(\nu)} < \infty$, $\liminf_k \norm{D_{N_k}}_{L^1(\nu)} < \infty$. Thus, we may apply Lemma $\ref{Komlos}$ to show that there exists a subsequence of $(D_{N_k})$ (which remain denoted as $(D_{N_k})$) such that the averages
		$\frac{1}{K} \sum_{k=1}^K D_{N_k}$
		converge $\nu$-a.e to a function $V \in L^1(\nu)$. Similarly, by Lemma \ref{L1} we have
		\[ V \circ \Phi(z) =  \lim_K \frac{1}{K} \sum_{k=1}^K D_{N_k} \circ \Phi(z) \]
		for $\nu$-a.e. $z \in X^H$. Therefore,
		\begin{align*}
		V - V \circ \Phi &= \lim_{K} \left( \frac{1}{K} \sum_{k=1}^K F - \frac{1}{K} \sum_{k=1}^K \left(\frac{1}{N_k} \sum_{n=1}^{N_k} F \circ \Phi^n \right) \right) = F.
		\end{align*}
		Since $V - V\circ \Phi = F$ for $\nu$-a.e., the equality certainly holds for $\mu_\Delta$-a.e. (the construction of $\nu$ guarantees that $V \in L^1(\mu_\Delta)$). 	
	\end{proof}
	
	\begin{cor} Let $(X, \mathcal{F}, \mu, T_1, \ldots, T_H)$ be a measure preserving system, and $f_1, f_2, \ldots f_H \in L^\infty(\mu)$ 
	The following statements are equivalent. 
	
	\begin{enumerate}
		\item The supremum of the nonconventional ergodic sums is $L^\infty$-bounded, i.e.
		\[ \sup_{N} \norm{\sum_{n= 1}^N \prod_{i=1}^H f_i \circ T_i^{n}  }_{L^\infty(\nu)} < \infty. \]
		\item The product of the functions is a coboundary in $L^\infty(X^H, \nu)$, i.e. if $\Phi = T_1 \times T_2 \times \cdots \times T_H$, there exists $V\in L^\infty(X^H, \nu)$ such that
		\[\bigotimes_{i=1}^H f_i = V - V \circ \Phi,  \text{$\nu$-a.e. }  
		 \]
		Therefore, for $\mu$-a.e. $x \in X$, we have
		\[f_1(x)f_2(x) \cdots f_H(x) = V(x, x, \ldots, x) - V(T_1x, T_2x, \ldots, T_Hx). \]
	\end{enumerate}
   \end{cor} 
   \begin{proof}
     The implication 1) implies 2) can be obtained by following  the same path as in the proof of the previous proposition. The only thing to check is that $ V \in L^{\infty}(\nu)$. But this follows from the fact that   
     if \[ \sup_{N} \norm{\sum_{n= 1}^N \prod_{i=1}^H f_i \circ T_i^{n}  }_{L^\infty(\nu)} < C<\infty \] then 
     $\| D_{N_k}\|_{L^{\infty}} \leq C.$  From this observation one can conclude that the limit function $V$ is also in $L^{\infty}(\nu)$.\\
     For the reverse implication, if  $F= f_1\times f_2...\times f_H $ is  a coboundary in $L^{\infty}(\nu)$ (i.e. $F = V - V\circ \Phi$ where $V\in L^{\infty}(\nu)$ )  then 
     $$\sum_{n =1}^N F\circ \Phi^n = V - V\circ \Phi ^{N+1}$$  and 
     $$\| \sum_{n =1}^N F\circ \Phi^n \|_{L^{\infty}(\nu)} = \|V - V\circ \Phi ^{N+1}\|_{L^{\infty}(\nu)} \leq 2 \| V \|_{L^{\infty}(\nu)}$$
   \end{proof}
    \noindent{\bf Remark}\\
     The assumption  \[ \sup_{N} \norm{\sum_{n= 1}^N \prod_{i=1}^H f_i \circ T_i^{n}  }_{L^p(\nu)} < \infty. \] is satisfied  when \[\sup_{N, m\in \ZZ}\norm{\sum_{n= 1}^N \prod_{i=1}^H f_i \circ T_i^{n +m}  }_{L^p(\mu)} < \infty. \]
    This last condition is equivalent to 
       \[\sup_{N\in \ZZ}\norm{\sum_{n= 0}^N \prod_{i=1}^H f_i \circ T_i^{n}  }_{L^p(\mu)} < \infty. \] which may be easier to check in the applications.\\

\end{document}